\documentclass[12pt]{amsart}

\usepackage{amsmath,amssymb,amsthm,enumerate}
\usepackage{ytableau}
\usepackage[margin=1in]{geometry}

\numberwithin{equation}{section}

\DeclareMathOperator{\bs}{bs}
\DeclareMathOperator{\Ht}{ht}

\newcommand{\calS}{\mathcal{S}}

\newtheorem{theorem}{Theorem}[section]

\newtheorem{corollary}[theorem]{Corollary}
\newtheorem{proposition}[theorem]{Proposition}
\newtheorem{lemma}[theorem]{Lemma}

\numberwithin{theorem}{section}
\title[Divisibility of character values of the symmetric group]{Divisibility of character values of the symmetric group by prime powers}
\author{Sarah Peluse and Kannan Soundararajan}

\address{Department of Mathematics, Stanford University, Stanford, CA, USA}
\email{speluse@stanford.edu}

\address{Department of Mathematics, Stanford University, Stanford, CA, USA}
\email{ksound@stanford.edu}
\dedicatory{In memory of Chandra Sekhar Raju}

\begin{document}

\begin{abstract}
  Let $k$ be a positive integer. We show that, as $n$ goes to infinity, almost every entry
  of the character table of $S_n$ is divisible by $k$. This proves a conjecture of Miller.
\end{abstract}
\maketitle
\section{Introduction}\label{sec:intro}
It is a standard fact that the irreducible characters of $S_n$ take only integer values
for every natural number $n$. In 2017, Miller~\cite{Miller2017} computed the character
tables of $S_n$ for all $n\leq 38$ and looked at various statistical properties of these
integers as $n$ grew. His computations suggested that
\begin{enumerate}
\item the density of even entries tended to $1$ as $n$ tended to infinity,
\item the density of entries divisible by $3$, the density of entries divisible by $5$,
  and the density of entries divisible by $7$ increase with $n$,
\item about half of the nonzero entries were positive,
\item and the density of zeros in the character table decreased as $n$ grew,
  but not very quickly.
\end{enumerate}
Based on this first observation, Miller~\cite{Miller2017,Miller2019} conjectured that as
$n$ goes to infinity, almost every entry of the character table of the symmetric group
$S_n$ is even. Following partial progress due to McKay~\cite{McKay1972},
Gluck~\cite{Gluck2019}, Ganguly, Prasad and Spallone~\cite{GPS20}, and Morotti~\cite{Morotti2020}, the first author proved this
conjecture in~\cite{Peluse2020}. Based on the second observation, Miller~\cite{Miller2017,Miller2019} also conjectured,
more generally, that for any fixed prime $p$, almost every entry of the character table of
$S_n$ is a multiple of $p$ as $n$ goes to infinity. We proved this conjecture
in~\cite{PeluseSoundararajan2022}, with a uniform upper bound for the number of entries
not divisible by a fixed prime. Recently, Miller~\cite{Miller2019ii} conjectured, even
more generally, that for any fixed prime power $q$, almost every entry of the character
table of $S_n$ is a multiple of $q$ as $n$ goes to infinity. In this paper, we prove this
most general of Miller's conjectures.

\begin{theorem}\label{thm:main}
  Let $n$ be large and $q\leq 10^{-3}\log{n}/(\log\log{n})^2$ be a prime power. The
  number of entries in the character table of $S_n$ that are not divisible by $q$ is at most
  \begin{equation*}
    O\left(p(n)^2\exp(-(\log \log n)^2) \right).
  \end{equation*}
\end{theorem}

It follows immediately from Theorem~\ref{thm:main} and the union bound that almost every
entry of the character table of $S_n$ is divisible by any fixed positive integer as $n$ goes to infinity.

\begin{corollary}\label{cor:main}
  Let $k$ be any positive integer. Then, as $n$ goes to infinity, the proportion of entries in the
  character table of $S_n$ that are not divisible by $k$ tends to $0$.
\end{corollary}

Our methods do not seem to shed any light on Miller's third and fourth observations. Most
interesting to us is the question of what proportion of character table entries are zero,
and recent large scale simulations of Miller and Scheinerman \cite{MS23} suggest that the proportion of zeros tends to $0$ as 
$n$ tends to infinity. Combining the Murnaghan--Nakayama rule and an old result of
Erd\H{o}s and Lehner~\cite{ErdosLehner1941} on the distribution of the largest part of a
uniformly random partition of $n$ produces a proportion of about $\frac{2}{\log{n}}$ zeros in
the character table of $S_n$, and no lower bound of a larger order of
magnitude seems to be known.  In the related setting of finite simple groups of Lie type, Larsen and
Miller~\cite{LarsenMiller2021} have shown that almost every character table entry is zero
as the rank goes to infinity.

\medskip

\noindent {\bf Acknowledgments.}  The first author was partially supported by the NSF
Mathematical Sciences Postdoctoral Research Fellowship Program under Grant No. DMS-1903038
and by the Oswald Veblen Fund.  The second author is partially supported by a grant from
the National Science Foundation, and a Simons Investigator Grant from the Simons
Foundation.   We thank David Speyer for drawing our attention to Lemma~\ref{lem:combiningparts}.   We are grateful to the referee for a 
careful reading, and for pointing out the reference \cite{MO88}, which contained our Lemma~\ref{lem6.1}. 

\section{Proof outline}\label{sec:outline}

For any partitions $\lambda$ and $\mu$ of $n$, let $\chi^\lambda_\mu$ denote the value of
the irreducible character of $S_n$ corresponding to $\lambda$ on the conjugacy class of
elements with cycle type corresponding to $\mu$. In~\cite{PeluseSoundararajan2022}, our
argument proceeded by combining two key facts: (i) if $\mu$ contains a part substantially
larger than the typical largest part of a random partition, then $\chi^{\lambda}_\mu=0$
for almost every $\lambda$, and (ii) if $\nu$ is another partition of $n$ that is obtained
from $\mu$ by combining $p$ parts of the same size $m$ into one part of size $pm$, then
$\chi^{\lambda}_{\mu}\equiv \chi^{\lambda}_{\nu}\pmod{p}$ for every $\lambda$. We showed
that, for almost every $\mu$, repeatedly combining $p$ parts of the same size in this
manner produces a partition $\widetilde{\mu}$ containing a very large part, large enough
so that $\chi^{\lambda}_{\widetilde{\mu}}$ must be zero for almost every $\lambda$. Our
main result on the divisibility of character values by primes then followed from the fact
that $\chi^\lambda_{\mu}\equiv\chi^{\lambda}_{\widetilde{\mu}}\pmod{p}$ for every
$\lambda$.

The second key fact generalizes to a congruence of character value modulo prime powers in
a straightforward manner.
\begin{lemma}\label{lem:combiningparts}
Let $p^r$ be a power of the prime $p$. Suppose that $\mu$ is a partition of $n$, and that
$\nu$ is another partition of $n$ obtained from $\mu$ by replacing $p^r$ parts of the same
size $m$ by $p^{r-1}$ parts of size $pm$. Then for all partitions $\lambda$ of $n$, we have
\begin{equation*}
 \chi^\lambda_\mu\equiv\chi^\lambda_{\nu}\pmod{p^r}.
\end{equation*}
\end{lemma}
However, when $r>1$, it is no longer the case that starting from a typical partition $\mu$
of $n$ and repeatedly combining $p^r$ parts of the same size $m$ into $p^{r-1}$ parts of
size $pm$ produces a partition $\widetilde{\mu}$ containing a part substantially larger than
the largest part of a typical partition of $n$. The argument
from~\cite{PeluseSoundararajan2022} that worked for primes thus breaks down for all other
prime powers.

The key idea used to overcome this barrier is a new condition for character values of
the symmetric group to be divisible by a fixed prime power, which we prove by exploiting
certain symmetries that appear after applying the Murnaghan--Nakayama rule multiple times.
\begin{theorem}\label{thm:partitiondivisibilitycondition}
  Let $n,m_1,\dots,m_r$ be distinct positive integers.  Let $\mu$ be
  a partition of $n$ containing parts of size $m_1,\dots,m_r$, each appearing at least
  $p^{r-1}$ times. If $\lambda$ is a $\left(\sum_{i=1}^rk_im_i\right)$-core partition of
  $n$ for all $r$-tuples $(k_1,\dots,k_r)$ of integers $0\leq k_1,\dots,k_r\leq p^{r-1}$
  for which some $k_i=p^{r-1}$, then
  \begin{equation*}
    p^r\mid \chi^\lambda_\mu.
  \end{equation*}
\end{theorem}

Starting with a partition $\mu$ of $n$, repeatedly combine 
$p^r$ parts of the same size $m$ into $p^{r-1}$ parts of size $pm$, until 
the process terminates in a partition $\widetilde{\mu}$ where no part appears more than $p^{r}-1$ times. 
As a preliminary to applying Theorem~\ref{thm:partitiondivisibilitycondition} we show that for a 
typical partition $\mu$, the resulting partition $\widetilde{\mu}$ will have $r$ parts that are 
suitably large, and with each appearing at least $p^{r-1}$ times.

% produces a partition
%$\widetilde{\mu}$ containing at least $r$ parts $m_1,\dots,m_r$ that all appear at least
%$p^{r-1}$ times for which $p^{r-1}m_1,\dots,p^{r-1}m_r$ are sufficiently large that almost
%every partition of $n$ is a simultaneous $\left(\sum_{i=1}^rk_im_i\right)$-core for all $r$-tuples
%$(k_1,\dots,k_r)$ as in .

\begin{proposition}\label{prop:manylargeparts}
  Starting with a partition $\mu$ of $n$, repeatedly replace every occurrence of $p^r$
  parts of the same size $m$ by $p^{r-1}$ parts of size $pm$ until we arrive at a
  partition $\tilde{\mu}$ where no part appears more than $p^r-1$ times. Then, except for
\begin{equation*}
  O\left(p(n)\exp\left(-n^{1/20p^r}\right)\right)
\end{equation*}
initial partitions $\mu$, the partition $\widetilde{\mu}$ contains at least $r$ distinct parts
$m_1,\dots,m_r$, each appearing at least $p^{r-1}$ times and satisfying
\begin{equation*}
  p^{r-1}m_j>\Big(1+\frac{1}{6p^r}\Big)\frac{\sqrt{6}}{2\pi}\sqrt{n}\log{n}.
\end{equation*}
This holds uniformly for $p^r\leq 10^{-3}\log{n}/(\log\log{n})^2$.
\end{proposition}

The significance of the lower bound on $p^{r-1}m_j$ in
Proposition~\ref{prop:manylargeparts} is that it lies beyond the threshold of values $t$
such that almost every partition of $n$ is a $t$-core.

\begin{lemma}\label{lem:tcores}
  Let $1\leq L\leq \log{n}/\log\log{n}$ be a real number. Then, for any given integer $t$ with 
  \begin{equation*}
    t\geq\Big(1+\frac{1}{L}\Big)\frac{\sqrt{6}}{2\pi}\sqrt{n}\log{n},
  \end{equation*}
  all but
  \begin{equation*}
    O\Big(p(n)\frac{\log{n}}{n^{1/2L}}\Big)
  \end{equation*}
  partitions of $n$ are $t$-cores.
\end{lemma}

We can swiftly deduce our main result,  Theorem~\ref{thm:main}, from the results stated above.

\begin{proof}[Deducing Theorem~\ref{thm:main}]
  Let $\mu$ be a partition of $n$, and suppose that $\widetilde{\mu}$ is as in
  Proposition~\ref{prop:manylargeparts}. Then, for all but at most
  \begin{equation*}
    O\left(p(n)\exp\left(-n^{1/20p^r}\right)\right)
  \end{equation*}
  choices of $\mu$, the partition $\widetilde{\mu}$ contains at least $r$ distinct parts
  $m_1,\dots,m_r$, each appearing at least $p^{r-1}$ times and satisfying
  \begin{equation}
  \label{2.1}
  p^{r-1}m_j>\Big(1+\frac{1}{6p^r}\Big)\frac{\sqrt{6}}{2\pi}\sqrt{n}\log{n}.
\end{equation}

Consider any $r$-tuple $(k_1, \ldots, k_r)$ with $0\le k_1, \ldots, k_r\le p^{r-1}$ and $k_i=p^{r-1}$ for 
some $i$.  Then $k_1 m_1 +\ldots +k_r m_r$ also exceeds  the bound in \eqref{2.1}, so that 
by Lemma~\ref{lem:tcores} all but $O(p(n) (\log n)/n^{\frac{1}{2L}})$ partitions $\lambda$ of $n$ 
are $(k_1m_1+\ldots+k_rm_r)$-cores.  Since there are at most $r (p^{r-1}+1)^{r-1}$ such $r$-tuples 
$(k_1,\ldots,k_r)$, by the union bound we see that all but at most 
\begin{equation*}
  O\left(p(n)\frac{\log{n}}{n^{1/12p^r}}r\left(p^{r-1}+1\right)^{r-1}\right)
\end{equation*}
partitions $\lambda$ of $n$ are $(k_1 m_1 +\ldots +k_r m_r)$-cores for all choices of the $r$-tuple $(k_1,\ldots, k_r)$.

Theorem~\ref{thm:partitiondivisibilitycondition} now shows that $p^r$ divides $\chi^{\lambda}_{\widetilde \mu}$, and since
$\chi^{\lambda}_{\mu}\equiv\chi^{\lambda}_{\widetilde{\mu}}\pmod{p^r}$ by
Lemma~\ref{lem:combiningparts}, it also follows that $p^r$ divides $\chi^{\lambda}_{\mu}$.   Putting everything together, 
we conclude that the number of partitions $\lambda$ and $\mu$ with $p^r\nmid \chi^{\lambda}_{\mu}$ is at most 
\begin{equation*}
  O\Big(p(n)^2 \Big( \exp(-n^{1/(20p^r)}) + \frac{1}{n^{1/13p^r}}r\left(p^{r-1}+1\right)^{r-1}\Big) \Big) 
  = O\Big( p(n)^2 \exp(-(\log \log n)^2)\Big), 
\end{equation*}
in the range $p^r \le 10^{-3} \log n/(\log \log n)^2$.  
\end{proof}

The rest of the paper is organized as follows.  We will prove
Lemmas~\ref{lem:combiningparts} and~\ref{lem:tcores} in Section~\ref{sec:lemmas},
Theorem~\ref{thm:partitiondivisibilitycondition} in
Sections~\ref{sec:abacus},~\ref{sec:primepower},~\ref{sec:MN}, and~\ref{seclem6.2}, and
Proposition~\ref{prop:manylargeparts} in Sections~\ref{sec:prelims}
and~\ref{sec:finalproof}.

\section{Proofs of Lemmas~\ref{lem:combiningparts} and~\ref{lem:tcores} }\label{sec:lemmas}
We begin by proving the two lemmas stated in the previous section.

\begin{proof}[Proof of Lemma~\ref{lem:combiningparts}]  We claim that  if $Q\in\mathbf{Z}[x_1,\dots,x_k]$ is a polynomial with integer coefficients,
  then
  \begin{equation*}
    Q(x_1,\dots,x_k)^{p^r}\equiv Q(x_1^p,\dots,x_k^p)^{p^{r-1}}\pmod{p^r}.
  \end{equation*}
As is well known, we may write 
 \begin{equation} \label{3.1} 
    Q(x_1,\dots,x_k)^p=Q(x_1^p,\dots,x_k^p)+p\cdot R(x_1,\dots,x_k)
  \end{equation}
 for some $R\in\mathbf{Z}[x_1,\dots,x_k]$, which establishes the claim when $r=1$.  
For $r>1$, raise both sides of \eqref{3.1} to the power $p^{r-1}$, and expand using the binomial theorem: 
  \begin{align*}
Q(x_1,\ldots,x_k)^{p^r} &=  \left(Q(x_1^p,\dots,x_k^p)+p\cdot R(x_1,\dots,x_k)\right)^{p^{r-1}} \\
&= 
Q(x_1^p, \ldots, x_k^p)^{p^{r-1}} + \sum_{\ell =1}^{p^{r-1}} \binom{p^{r-1}}{\ell} Q(x_1^p,\ldots, x_k^p)^{p^{r-1}-\ell} (pR(x_1,\ldots, x_k))^{\ell}. 
 \end{align*}
 Note that for $1\le \ell \le p^{r-1}$ 
   \begin{equation*}
    p^{\ell}\binom{p^{r-1}}{\ell}= p^{\ell}\frac{p^{r-1}}{\ell}\binom{p^{r-1}-1}{\ell-1} \equiv 0 \pmod{p^r}, 
  \end{equation*} 
  since the power of $p$ dividing $\ell$ is certainly at most $\ell-1$.  This establishes our claim. 
  
  The lemma now follows by applying this observation to the polynomials appearing in Frobenius's formula for the character values
  $\chi^\lambda_\mu$ and $\chi^{\lambda}_{\nu}$ (see Chapter 4 of \cite{FultonHarris1991}).
\end{proof}
\begin{proof}[Proof of Lemma~\ref{lem:tcores}]
  The proof is essentially identical to that of Proposition~1
  of~\cite{PeluseSoundararajan2022}, but we include the short argument for
  completeness. Since every partition of $n$ is a $t$-core for $t>n$, we may naturally assume that $t\le n$.  
  From Lemma~5 of~\cite{Morotti2020}, we know that at most $(t+1)p(n-t)$
  partitions of $n$ are not $t$-cores. By the asymptotic formula
  \begin{equation*}
    p(m)\sim\frac{1}{4\sqrt{3}m}\exp\Big(\frac{2\pi}{\sqrt{6}}\sqrt{m}\Big)
  \end{equation*}
  for the partition function, we have
\begin{equation*}
  (t+1)p(n-t)\ll\frac{t+1}{n-t+1}\exp\left(\frac{2\pi}{\sqrt{6}}\sqrt{n-t}\right)\leq\frac{t+1}{n-t+1}\exp\left(\frac{2\pi}{\sqrt{6}}\sqrt{n}-\frac{\pi
    t}{\sqrt{6n}}\right).
\end{equation*}
In the range $n\geq t \ge (1+1/L) \frac{\sqrt{6}}{2\pi} \sqrt{n}\log n$, the
right-hand side above is maximized at the lower endpoint $t=\left(1+1/L\right)\frac{\sqrt{6}}{2\pi}\sqrt{n}\log{n}$. It follows that the number of
 partitions of $n$ that are not $t$-cores is
\begin{equation*}
  \ll \frac{\log{n}}{\sqrt{n}}n^{-(1+1/L)/2}\exp\Big(\frac{2\pi}{\sqrt{6}}\sqrt{n}\Big)\ll p(n)\frac{\log{n}}{n^{1/2L}},
\end{equation*}
where the last step uses again the asymptotic for the partition function.
\end{proof}

\section{Partitions and Abaci}\label{sec:abacus}

The proof of Theorem~\ref{thm:partitiondivisibilitycondition} requires the machinery 
of the {\sl abacus} associated to a partition.   A  good reference for this theory is Section~2.7 
of~\cite{JamesKerber1981}, and we recall some salient facts below.  

\subsection{The notion of an abacus}  An abacus is a bi-infinite sequence of $0$'s and $1$'s beginning with an infinite 
sequence of $1$'s and ending with an infinite sequence of $0$'s.   

More formally,  let
\begin{equation*}
  \calS:=\left\{s:\mathbf{Z}\to\{0,1\}:\text{there exists a }k\geq 0\text{ such that }s(-i)=1\text{ and }s(i)=0\text{ for all }i\geq k\right\}
\end{equation*}
denote the set of all sequences of $0$'s and $1$'s indexed using the integers, that begin with an infinite
sequence of $1$'s and end with an infinite sequence of $0$'s. For example,
\begin{equation*}
 \dots,1,\dots,1,1,1,0,0,1,0,1,1,0,0,0,\dots,0,\dots
\end{equation*}
is in $\calS$.  We consider two sequences $s$ and $s^{\prime}$ in ${\mathcal S}$ to be
equivalent if there is some integer $j$ such that $s(i) = s^{\prime}(i-j)$ for all $i$,
that is, if $s^{\prime}$ can be produced by shifting the terms in $s$ by $j$.  This is an
equivalence relation, and an abacus refers to an equivalence class in ${\mathcal S}$ under
this relation.  We denote by ${\mathcal A}$ the set of such abaci, so that by an element
$a$ of ${\mathcal A}$ we mean the equivalence class consisting of some sequence
$s\in {\mathcal S}$ together with all its shifts.

%Let $R:\calS\to\calS$ denote the right shift operator on $\calS$, so that
%\begin{equation*}
 % Rs:=\left(s(i-1)\right)_{i}
%\end{equation*}
%for all $s\in\calS$, and for any integer $j$, denote by $R^{j}$ 

%and define an equivalence relation $\sim$ on $\calS$ by
%\begin{equation*}
 % s\sim s'\qquad \iff \qquad R^{j}s=s'\text{ or }s=R^js'\text{ for
 %   some }j\geq 0.
%\end{equation*}
%The set $\mathcal{A}$ of abacus representations of partitions is defined to be the set of
%equivalence classes of this relation:
%\begin{equation*}
%  \mathcal{A}:=\calS/\sim.
%\end{equation*}

\subsection{The abacus associated to a partition}  We now show how abaci are  in one-to-one correspondence with partitions of integers.  Starting 
with an integer partition $\lambda$, we construct an abacus $a_{\lambda} \in {\mathcal A}$ as follows.  Draw the Young diagram of $\lambda$, 
and trace out the boundary of the diagram, moving from the lower left-hand corner to the upper right-hand corner, writing a $0$ for each horizontal 
move and a $1$ for each vertical move.  Then prepend an infinite string of $1$'s and append an infinite string of $0$'s to find a
representative of the corresponding element $a_\lambda$ of $\mathcal{A}$.   

This procedure is easily reversed, and starting with an abacus $a$ in ${\mathcal A}$ 
we obtain a Young diagram, which corresponds to a partition $\lambda$.  If $s \in {\mathcal S}$ is a 
representative of $a$, then the partition $\lambda$ is a partition of the integer $n(a)$ which 
counts the number of pairs of indices $(i,j)$ with $i<j$ such that $s(i)=0$ and $s(j)=1$.  

%  For any $a\in\mathcal{A}$, define $n(a)$ to be the number of pairs of indices $(i,j)$,
%$i<j$, for which $s(i)=0$ and $s(j)=1$. The elements of $\mathcal{A}$ are in one-to-one
%correspondence with integer partitions, with each $a\in\mathcal{A}$ corresponding to a
%partition of $n(a)$. Given an integer partition $\lambda$, we can produce a sequence of
%$0$'s and $1$'s by drawing the Young diagram of $\lambda$ and tracing out the boundary of
%the diagram, moving from the lower left-hand corner to the upper right-hand corner,
%writing a $0$ for each horizontal move and a $1$ for each vertical move. We 

To illustrate, consider the partition $(6,5,3,1,1,1)$, whose Young diagram is pictured in
Figure~\ref{fig:653111}.
\begin{figure}
\begin{equation*}
\ydiagram{6,5,3,1,1,1}
\end{equation*}
  \caption{The Young diagram of $(6,5,3,1,1,1)$}\label{fig:653111}
\end{figure}
If we start in the lower left-hand corner of this diagram and move along the boundary to
the upper right-hand corner, we move right, up three times, right twice, up, right twice, up,
right, and up. The correspondence described above produces the string
\begin{equation}
  \label{eq:seq}
  0,1,1,1,0,0,1,0,0,1,0,1,
\end{equation}
which we can turn into a bi-infinite sequence by adding an infinite sequence of $1$'s to
the beginning and an infinite sequence of $0$'s to the end:
\begin{equation}
  \label{eq:rep}
  \dots,1,\dots,1,1,1,0,1,1,1,0,0,1,0,0,1,0,1,0,0,0,\dots,0,\dots.
\end{equation}
The equivalence class of this sequence is the abacus associated to $(6,5,3,1,1,1)$.

\subsection{Hooks and border strips} Let $\lambda$ be a partition.  The {\sl hook} $h$
associated to a box $b$ in the Young diagram of $\lambda$ consists of the box $b$ together
with all the boxes directly to its right and directly below it.  The hook-length of $h$,
denoted by $\ell(h)$, is the number of boxes contained in the hook.  The {\sl height} of
the hook $h$, denoted by $\Ht(h)$, is one less than the number of rows in the Young
diagram of $\lambda$ that contain a box of $h$.  Associated to each hook is a {\sl border
  strip} (also known as a skew hook), denoted ${\text {bs}}(h)$, which is the connected
region of boundary boxes of the Young diagram running from the rightmost to the bottommost
box of $h$. Removing such a border strip leaves behind a smaller Young diagram.  These
notions play a prominent role in the representation theory of the symmetric group, and in particular
feature in the Murnaghan--Nakayama rule for computing character values, which we next
recall (see Theorem 2.4.7 of ~\cite{JamesKerber1981}, and also Chapter 4 of
\cite{FultonHarris1991}).

\begin{theorem}[The Murnaghan--Nakayama rule]
  Let $n$ and $t$ be positive integers, with $t\leq n$. Let $\sigma\in S_n$ be of the form
  $\sigma=\tau\cdot \rho$, where $\rho$ is a $t$-cycle, and $\tau$ is a permutation of $S_n$ 
  with support disjoint from $\rho$. Let $\lambda$ be a partition of $n$. Then
  \begin{equation}
    \label{eq:MN}
    \chi^\lambda(\sigma)=\sum_{\substack{h\in\lambda \\ \ell(h)=t}}(-1)^{\Ht(h)}\chi^{\lambda\setminus\bs(h)}(\tau).
  \end{equation}
 \end{theorem}

 Above, $\chi^{\lambda}(\sigma)$ denotes the value of the character of the irreducible
 representation of $S_n$ corresponding to the partition $\lambda$, evaluated on the
 conjugacy class of $\sigma$, $\lambda\setminus \bs(h)$ denotes the partition of $n-t$
 obtained by removing the border strip $\bs(h)$ from the Young diagram of $\lambda$, and
 $\chi^{\lambda\setminus \bs(h)}(\tau)$ denotes the character value of the irreducible
 representation of $S_{n-t}$ corresponding to the partition $\lambda\setminus \bs(h)$
 evaluated on the conjugacy class of $\tau$.

 The abacus notation helps with thinking about hook lengths and border strips.  Let
 $\lambda$ be a partition, let $a_\lambda$ denote the corresponding abacus, and let $s$ be
 a representative in ${\mathcal S}$ for the abacus $a_\lambda$.  Each hook $h$ in the
 Young diagram of $\lambda$ is in natural one-to-one correspondence with a pair of indices
 $(i,j)$, $i<j$, with $s(i)=0$ and $s(j)=1$.  The length of the hook $h$ is $j-i$.  In
 particular, the partition $\lambda$ contains no hooks of length $t$ (that is, $\lambda$
 is a $t$-core) if and only if there is no pair of indices $(i,i+t)$ with $s(i)=0$ and
 $s(i+t)=1$.  The height of the hook $h$ equals the number of $1$'s in the sequence $s$
 lying strictly between the $0$ at index $i$ and the $1$ at index $j$:
\begin{equation*}
  \Ht(h)=\#\left\{i<k<j:s(k)=1\right\}.
\end{equation*}

Further, the abacus notation gives an easy description of the result of removing a border
strip from a partition.  Define, for any pair of distinct integers $(i,j)$ the operator
$T_{ij}: {\mathcal S} \to {\mathcal S}$ that swaps the terms indexed by $i$ and $j$ in a
bi-infinite sequence $s\in {\mathcal S}$ and leaves all other entries fixed.  Thus for
$s\in {\mathcal S}$
$$ 
(T_{ij} s) (k) =   \begin{cases}
    s(k) & k\neq i, j \\
    s(j)    & k=i \\
    s(i)   & k=j.
  \end{cases}
$$
With this notation in place, suppose $\lambda$ is a partition, and $s\in a_{\lambda}$ is a representative of the abacus of $\lambda$.  
Let $h$ be a hook of $\lambda$, corresponding to the pair of indices $(i,j)$ (with $i<j$) in $s$.   Then $T_{ij}s$ is a representative of 
the abacus associated to $\lambda \setminus \bs(h)$. 

%The set of hooks
%of length $t$ in the Young diagram of $\lambda$ is in one-to-one correspondence with such
%pairs of indices $(i,j)$ satisfying $j-i=t$. In particular, $\lambda$ is a $t$-core if no
%such pair of indices exists.  Removing the corresponding border strip of length $t$ from
%the Young diagram of $\lambda$ corresponds to swapping a $0$ with a $1$ that is $t$
%indices to the right. That is, if $s\in a_{\lambda}$ and the hook $h$ in the Young diagram
%of $\lambda$ corresponds to the border strip $\bs(h)$ and to the pair of indices $(j,j+t)$
%with $s(j)=0$ and $s(j+t)=1$, then $s'\in a_{\lambda\setminus \bs(h)}$, where
%\begin{equation*}
 % s'(i):=
%\end{equation*}
%More generally, for any pair of distinct integers $(i,j)$, define $T_{ij}:\calS\to\calS$ by
%\begin{equation*}
 % (T_{ij}s)(k)=
 % \begin{cases}
  %  s(k) & k\neq i,j \\
   % s(j) & k=i \\
   % s(i) & k=j
  %\end{cases},
%\end{equation*}
%so that $T_{ij}$ swaps the terms indexed by $i$ and $j$ in any bi-infinite sequence. Then
%$s'=T_{j,j+t}s$ and $T_{j,j+t}s'=s$.

Returning to our example of the partition $(6,5,3,1,1,1)$, Figure~\ref{fig:653111hooks}
contains its Young diagram again, but now with each box filled in with the
corresponding hook-length.
\begin{figure}
\begin{equation*}
  \ytableausetup{centertableaux}
  \begin{ytableau}
    11 & 7 & 6 & 4 & 3 & 1 \\
    9 & 5 & 4 & 2 & 1 \\
    6 & 2 & 1 \\
    3 \\
    2 \\
    1
  \end{ytableau}
\end{equation*}
\caption{Hook-lengths for $(6,5,3,1,1,1)$}\label{fig:653111hooks}
\end{figure}
The unique hook $h$ of length $5$ in the diagram corresponds to the pair of indices
$(5,10)$ of the sequence~\eqref{eq:seq}. If we remove the corresponding border strip, we
obtain the diagram pictured in Figure~\ref{fig:621111}, which corresponds to the partition
$(6,5,3,1,1,1)\setminus \bs(h)=(6,2,1,1,1,1)$ and the bi-infinite sequence
\begin{equation*}
  \dots,1,\dots,1,1,1,0,1,1,1,1,0,1,0,0,0,0,1,0,0,0\dots,0,\dots
\end{equation*}
of $0$'s and $1$'s.
\begin{figure}
\begin{equation*}
  \ydiagram{6,2,1,1,1,1}
\end{equation*}
\caption{The Young diagram of $(6,2,1,1,1,1)$}\label{fig:621111}
\end{figure}
Note that if we swap the $0$ and $1$ corresponding to the hook $h$
in the representative~\eqref{eq:rep} of $a_{(6,5,3,1,1,1)}$, then we get an equivalent bi-infinite
sequence.

\subsection{Removing several hooks in succession}  In our work below, we will need to remove several hooks (more precisely, the border strips corresponding 
to those hooks) 
in succession from a partition.  By removing a sequence of hooks $h_1$, $\ldots$, $h_R$ from a partition $\lambda$, we 
mean the following: $h_1$ is a hook of $\lambda$, $h_2$ is a hook of $\lambda \setminus \bs(h_1)$, $h_3$ is a 
hook of $\lambda\setminus \bs(h_1) \setminus \bs(h_2)$, and so on, until we arrive at $h_R$ 
which is a hook of $\lambda \setminus \bs(h_1) \ldots \setminus \bs(h_{R-1})$, and when this is removed we obtain 
the final partition $\lambda^{\prime} = \lambda \setminus \bs(h_1) \ldots \setminus \bs (h_R)$.    

Let $s$ be a representative of the abacus $a_\lambda$ associated to $\lambda$.  Let
$(i_1,j_1)$ denote the pair of indices in $s$ corresponding to the hook $h_1$, $(i_2,j_2)$
the corresponding pair to $h_2$ (which, recall, is a hook of $\lambda \setminus \bs(h_1)$
corresponding to the bi-infinite sequence $T_{i_1,j_1}s$), and so on.  Thus, the sequence
of hooks $h_1$, $\ldots$, $h_R$ may be encoded by the $R$-tuple of pairs
$((i_1,j_1),(i_2, j_2), \ldots, (i_R,j_R))$, and the process of removing these hooks
results in the sequence
$$ 
s^{\prime} = T_{i_{R},j_{R}}  T_{i_{R-1}, j_{R-1} } \cdots T_{i_1, j_1} s. 
$$ 
The sequence $s^{\prime}$ is a representative of the abacus $a_{\lambda^{\prime}}$ associated to the partition $\lambda^{\prime}$. 

Of particular interest for us will be the situation where all the hooks have the same
length, $m$ say.  Here $j_k = i_k +m$ for all $1\le k\le R$, and we may encode the
sequence of hooks by simply the $R$-tuple $(i_1, \ldots, i_R)$.  Note that the indices
$i_1$, $\ldots$, $i_R$ may contain repeats, but there are also constraints, such as
$i_2 \neq i_1$ (since $(i_1,i_1+m)$ is a hook in $s$ and so it cannot be a hook in
$T_{i_1, i_1+m} s$).

\section{Plan of the proof of Theorem~\ref{thm:partitiondivisibilitycondition}}\label{sec:primepower}

We begin by restating Theorem~\ref{thm:partitiondivisibilitycondition} in terms of values of
irreducible characters at elements of $S_n$, which will make the notation involved in its
proof cleaner.

\begin{theorem}[An equivalent formulation of Theorem~\ref{thm:partitiondivisibilitycondition}]
  \label{thm:divisibilitycondition}
  Let $n$, $m_1$, $\dots$, $m_{r}$ be distinct positive integers.   Let 
  $\sigma\in S_n$ be a permutation of the form
  \begin{equation*}
    \sigma=\tau\cdot\prod_{i=1}^{r}\prod_{j=1}^{p^{r-1}}\rho_i^{(j)},
  \end{equation*}
  where each $\rho_i^{(j)}$ is a cycle of length $m_i$, the supports of all the cycles
  $\rho_i^{(j)}$ are disjoint, and $\tau\in S_n$ is a permutation with support disjoint
  from those of the $\rho_i^{(j)}$'s.   Suppose that $\lambda$ is a $(\sum_{i=1}^{r}k_im_i)$-core
  partition of $n$ for all $r$-tuples $(k_1,\dots,k_{r})$ of integers
  $0\leq k_1,\dots,k_r\leq p^{r-1}$ for which some $k_i=p^{r-1}$.  Then
  \begin{equation*}
    p^r\mid \chi^\lambda(\sigma).
  \end{equation*}
\end{theorem}

The proof of Theorem~\ref{thm:divisibilitycondition} rests on the following crucial proposition, which is 
based on applying the Murnaghan--Nakayama rule $p^{r-1}$ times.   

\begin{proposition}
  \label{lem:pm-1}  Let $r$, $m$ and $n$ be positive integers.  Let $\sigma\in S_n$ be of the form
  \begin{equation*}
    \sigma=\tau\cdot\prod_{j=1}^{p^{r-1}}\rho^{(j)},
  \end{equation*}
  where each $\rho^{(j)}$ is an $m$-cycle, with all the cycles $\rho^{(j)}$ being disjoint, and with $\tau\in S_n$ being a permutation whose 
  support is disjoint from all the cycles $\rho^{(j)}$.  
  Denote by $L$ the set of partitions of $n-p^{r-1}m$ that can be obtained from $\lambda$ by removing, in succession, $p^{r-1}$
  border strips of length $m$. If $\lambda$ is a $p^{r-1}m$-core partition of $n$,
  then
  \begin{equation*}
    \chi^\lambda(\sigma)=p\sum_{\lambda'\in L}\epsilon_{\lambda'}\chi^{\lambda'}(\tau),
  \end{equation*}
  where each $\epsilon_{\lambda'}$ is an integer.
\end{proposition}

We will quickly deduce Theorem~\ref{thm:divisibilitycondition} (and hence Theorem~\ref{thm:partitiondivisibilitycondition}) 
from Proposition~\ref{lem:pm-1} and the following simple observation.  
 
\begin{lemma}
  \label{lem:pm-1cores}  Let $n$, $t$ and $m$ be positive integers.  Let $\lambda$ be a partition of $n$ 
  which is both a $t$-core and a $(t+m)$-core.   Let $\lambda^{\prime}$ be a partition of $n-m$ that can be obtained 
  by removing a border strip of length $m$ from $\lambda$.  Then $\lambda^{\prime}$ is a $t$-core.  
\end{lemma}
\begin{proof}  If $\lambda$ has no hook (and thus no border strip) of length $m$ then the lemma holds vacuously.  Suppose 
that $\lambda^{\prime}$ arises from removing the border strip corresponding to the hook $h$ of length $m$ in $\lambda$.  
Let $a_\lambda$ be the abacus of $\lambda$, and $s$ be a representative bi-infinite sequence in $a_\lambda$.  Suppose 
the hook $h$ corresponds to the pair of indices $(i, i+m)$ with $s(i)= 0$ and $s(i+m)=1$, so that the partition $\lambda^{\prime}$ 
corresponds to the abacus containing ${s'} = T_{i,i+m}  s$.  

If $\lambda^{\prime}$ is not a $t$-core, then there must exist a pair of indices $(j, j+t)$ with ${s'}(j)=0$ and ${s'}(j+t)= 1$.  
Since the entries of $s$ and $s'$ differ only at the indices $i$ and $i+m$, and since $\lambda$ is a $t$-core, we must have either 
$j=i+m$, or $j+t=i$.  If $j=i+m$, then $s(i)=0$ and $s(i+t+m)={s'}(j+t)=1$ which contradicts the assumption that $\lambda$ is a $(t+m)$-core. 
If $j=i-t$, then $s(i-t) ={s'}(j) = 0$ and $s(i+m)=1$, which again contradicts the assumption that $\lambda$ is a $(t+m)$-core.  
\end{proof}

%We will first show how Theorem~\ref{thm:divisibilitycondition} follows from these lemmas,
%and then give their proofs.

\begin{proof}[Deducing Theorem~\ref{thm:divisibilitycondition} from Proposition~\ref{lem:pm-1}]  Apply Proposition~\ref{lem:pm-1} 
first with $m=m_r$ to obtain 
$$ 
\chi^{\lambda}(\sigma)  = p \sum_{\lambda^{\prime} \in L} \epsilon_{\lambda^{\prime}} \chi^{\lambda^{\prime}} \Big( \tau \prod_{i=1}^{r-1}\prod_{j=1}^{p^{r-1}} \rho_i^{(j)} \Big). 
$$ 
If $t$ is any number of the form $t=\sum_{i=1}^{r-1} k_i m_i$ where the $k_i$ lie in
$[0,p^{r-1}]$ with at least one of them being $p^{r-1}$, then $\lambda$ is a
$(t+k_r m_r)$-core for all $0\le k_r \le p^{r-1}$.  Since any $\lambda^{\prime}\in L$
arises from $\lambda$ by removing $p^{r-1}$ border strips of length $m_r$, it follows by
$p^{r-1}$ applications of Lemma~\ref{lem:pm-1cores} that $\lambda^{\prime}$ is a $t$-core.

We may now repeat this argument, applying Proposition~\ref{lem:pm-1} to each
$\lambda^{\prime}\in L$ and now removing $p^{r-1}$ border strips of length $m_{r-1}$.
Applications of Lemma~\ref{lem:pm-1cores} show that the new partitions
$\lambda^{\prime\prime}$ that arise are ($\sum_{i=1}^{r-2} k_i m_i$)-cores for all choices
of $0\le k_i\le p^{r-1}$ with some $k_i=p^{r-1}$.

Carrying this argument out $r$ times, we obtain the desired result.  
\end{proof}

The proof of Proposition~\ref{lem:pm-1} depends on the following two lemmas, which we shall prove in the next two sections. 

\begin{lemma} \label{lem6.1}  Let $\lambda$ be a partition, and let $\lambda^{\prime}$ be 
obtained from $\lambda$ by removing a sequence of $R$ border strips of the same length $m$.   Let $h_1$, $\ldots$, $h_R$ be a sequence of $R$ hooks of length $m$  
which may be removed from the initial partition $\lambda$ to result in the final partition $\lambda^{\prime}$.   Then 
$$ 
(-1)^{\Ht (h_1) + \ldots + \Ht(h_R)} = \epsilon(\lambda, \lambda^{\prime})
$$ 
where the sign $\epsilon(\lambda, \lambda^{\prime}) = \pm 1$ depends only on the initial and final partitions $\lambda$ and $\lambda^{\prime}$ and is the same 
for all such possible sequences of hooks.  
\end{lemma} 

We are grateful to the referee for pointing out that Lemma~\ref{lem6.1} may be found in the literature as Proposition 2.2 of \cite{MO88}.  In the interest of 
keeping our exposition self-contained, we include the short proof of Lemma \ref{lem6.1} in Section 6.

\begin{lemma} \label{lem6.2} Let $\lambda$ be a $p^{r-1}m$-core partition, and let
  $\lambda^{\prime}$ be a partition that can be obtained from $\lambda$ by removing
  $R=p^{r-1}$ border strips of length $m$.  The number of tuples $(i_1, \ldots, i_R)$ such
  that
$$ 
s^{\prime} = T_{i_R, i_R+m} T_{i_{R-1},i_{R-1}+m} \cdots T_{i_1, i_1+m} s 
$$ 
is a multiple of $p$.  Here $s$ is a representative of the abacus of $\lambda$, and the partition $\lambda^{\prime}$ corresponds to the 
abacus containing $s^{\prime}$.
\end{lemma}

Once Lemmas \ref{lem6.1} and \ref{lem6.2} are in place, it is a simple matter to deduce Proposition~\ref{lem:pm-1}.

\begin{proof}[Deducing Proposition~~\ref{lem:pm-1}]  We apply the Murnaghan--Nakayama rule repeatedly while removing in succession 
$R=p^{r-1}$ hooks of length $m$ from $\lambda$.  This will result in an expression for $\chi^{\lambda}(\sigma)$ of the form 
$\sum_{\lambda^{\prime} \in L} c_{\lambda^{\prime}} \chi^{\lambda^{\prime}}(\tau)$, for suitable integers $c_{\lambda^{\prime}}$ 
which we must show are multiples of $p$.  Now 
$$ 
c_{\lambda^{\prime}} = \sum_{(i_1, \ldots, i_R)} (-1)^{\Ht(h_1) +\ldots + \Ht(h_R)} 
$$ 
where the sum is over all $R$-tuples $(i_1,\ldots, i_R)$ corresponding to hooks $h_1$, $\ldots$, $h_R$, which when removed from $\lambda$ 
in order result in the partition $\lambda^{\prime}$.   Lemma \ref{lem6.1} tells us that the sign $(-1)^{\Ht(h_1) +\ldots + \Ht(h_R)}$ is the same 
for all suitable tuples $(i_1,\ldots,i_R)$, and Lemma \ref{lem6.2} tells us that the number of such $R$-tuples is a multiple of $p$.  
\end{proof}

%  It suffices to show that, for every $\lambda'\in L$ and $\epsilon\in\{0,1\}$, the number
 % of $p^{r-1}$-tuples of hooks $(h_1,\dots,h_{p^{r-1}})$ of length $m$ with
  %\begin{equation*}
   % h_1\in\lambda,h_2\in\lambda\setminus\bs(h_1),\dots,h_{p^{r-1}}\in\lambda\setminus\bs(h_1)\setminus\dots\setminus\bs(h_{p^{r-1}-1})
  %\end{equation*}
 % such that
 % \begin{equation*}
  %  \lambda'=\lambda\setminus\bs(h_1)\setminus\dots\setminus\bs(h_{p^{r-1}})
 % \end{equation*}
 % and
 % \begin{equation*}
  %  \sum_{i=1}^{p^{r-1}}\Ht(h_i)\equiv \epsilon\pmod{2}
 % \end{equation*}
  %is a multiple of $p$. Indeed, this would tell us that each signed character value
 % \begin{equation*}
 %   (-1)^{\sum_{i=1}^{p^{r-1}}\Ht(h_i)}\chi^{\lambda\setminus\bs(h_1)\setminus\dots\setminus\bs(h_{p^{r-1}})}(\tau)
 % \end{equation*}
 % occurs a multiple of $p$ times as we range over all possible ways to remove $p^{r-1}$
 % hooks of length $m$ from $\lambda$, and so, since
 % \begin{equation*}
  %  \chi^{\lambda}(\sigma)=\sum_{\substack{h_1\in\lambda,h_2\in\lambda\setminus\bs(h_1),\dots,h_{p^{r-1}}\in\lambda\setminus\bs(h_1)\setminus\dots\setminus\bs(h_{p^{r-1}-1})
   %     \\ \ell(h_1)=\dots=\ell(h_{p^{r-1}})=m}}(-1)^{\sum_{i=1}^{p^{r-1}}\Ht(h_i)}\chi^{\lambda\setminus\bs(h_1)\setminus\dots\setminus\bs(h_{p^{r-1}})}(\tau)
%  \end{equation*}
 % by $p^{r-1}$ applications of the Murnaghan--Nakayama rule, the conclusion of the
  %lemma follows.

\section{Parity of heights of hooks:  Proof of Lemma~\ref{lem6.1}}\label{sec:MN}

Let $\lambda$ be a partition, and $s$ a representative of the abacus $a_\lambda$ associated 
to $\lambda$.   Augment $s$ by coloring a finite number $N$ 
of $1$'s in $s$ with distinct colors, taking care to color all the $1$'s appearing to the right of the first zero in $s$.  The $1$'s appearing to the left of the first $0$ 
are unimportant, but we allow the flexibility of coloring some of them since this situation may arise at an intermediate step when we remove hooks from $\lambda$.  
Note that the number 
of  $1$'s appearing to the right of the first zero equals the number of rows in the partition $\lambda$.  Thus $N$ is 
at least the number of rows in $\lambda$.  Color these $1$'s in the order of their appearance in $s$ using the colors $c_1$, $\ldots$, $c_N$. Call the 
augmented sequence ${\widehat s}$.  

 We begin with a general observation on removing hooks.  Suppose $(i,j)$ is a pair of indices corresponding to a hook $h$ in $s$ (at the moment 
the hook can have any length $j-i$).  Removing this hook produces the sequence $T_{i,j} s$.  Considering the augmented sequence ${\widehat s}$, we 
have the corresponding augmented sequence $T_{i,j} {\widehat s}$ after removing this hook.  If we consider the sequence of colors among the 
$1$'s in this sequence, we obtain a permutation $\pi_{ij}$, say, of the original sequence of colors $(c_1, \ldots, c_N)$ --- the $1$ appearing in $(T_{i,j}{\widehat s})(i)$ has 
the color of the $1$ in ${\widehat s}(j)$, and all other $1$'s in $T_{ij}({\widehat s})$ retain their color in ${\widehat s}$.  If the height of the hook removed is $k$, then note that 
${\widehat s}$ had $k$ colored $1$'s between $s(i)=0$ and $s(j)=1$ and the permutation $\pi_{ij}$ can be obtained by making $k$-transpositions, each time swapping the color of 
the $1$ at position $j$ by the color immediately preceding it.  Thus $(-1)^{k}= (-1)^{\Ht(h)}$ equals the sign of the permutation $\pi_{ij}$.

If we remove hooks $h_1$, $\ldots$, $h_\ell$ in succession (again, their lengths could be arbitrary), then the associated permutations of colors multiply, and 
therefore so do the signs of these permutations.  Thus, after removing these hooks in succession we would arrive at a permutation $\pi$ of the sequence of colors 
$(c_1, \ldots, c_N)$ and 
$$ 
(-1)^{\Ht(h_1) + \Ht(h_2) +\ldots + \Ht(h_\ell)} = \text{sgn}(\pi).
$$ 
 
 We now turn to the situation of the lemma, where a sequence $h_1$, $\ldots$, $h_R$ of $R$ hooks is removed all of length $m$.  Our observation 
 above shows that removing these hooks in order leads to the sequence ${\widehat s}^{\, \prime}$ where the color of the $1$'s is given by a permutation $\pi$ of the original 
 sequence of colors $c_1$, $\ldots$, $c_N$.  Further the sign of this permutation $\text{sgn}(\pi)$ equals $(-1)^{\Ht(h_1)+ \ldots + \Ht(h_R)}$. 
 
 To complete the proof, we will show that every way of removing $R$ hooks of length $m$ that leads to the partition $\lambda^{\prime}$ results in the same permutation of 
 colors $\pi$.  Consider the subsequence of ${\widehat s}$ obtained by restricting to a progression $(\bmod \, m)$: namely, $({\widehat s}(a+\ell m))_{\ell\in {\mathbf Z}}$. 
 There are $m$ such subsequences corresponding to $a=1$, $\ldots$, $m$.  Since the hooks removed all have length $m$, each removal of a hook affects only the terms within 
 one of these subsequences, leaving all the other subsequences unaltered.  Further within any particular subsequence  $({\widehat s}(a+\ell m))_{\ell\in {\mathbf Z}}$, 
 it is impossible to alter the original sequence of colors by removing any sequence of hooks of length $m$.  Therefore we can determine uniquely the color of 
 any element in ${\widehat s}^{\, \prime}$:  the $1$'s appearing in this sequence in the progression $a \pmod m$ have colors determined by their order of appearance in the original sequence $s$.

%\begin{lemma}\label{lem:heightmod2}
 % Let $n,R,$ and $m$ be positive integers with $Rm\leq n$, $\lambda$ be a partition of
 % $n$, $s\in a_\lambda$, and $h_1,\dots,h_R$ be hooks of length $m$ for which
%\[
%  h_1\in\lambda,h_2\in\lambda\setminus\bs(h_1),\dots,h_R\in\lambda\setminus\bs(h_1)\setminus\dots\setminus\bs(h_{R-1}),
 % \]
%  with corresponding pairs of indices $(j_1,j_1+m),\dots,(j_R,j_R+m)$, respectively. Set
 % $\lambda':=\lambda\setminus\bs(h_1)\setminus\dots\setminus\bs(h_R)$ and
 % \begin{equation*}
 %   s':=T_{j_R,j_R+m}\cdots T_{j_1,j_1+m}s,
 % \end{equation*}
%  so that $s'\in a_{\lambda'}$. Let $\sigma\in S_R$ be such that
 % \begin{equation*}
  %  s'=T_{j_{\sigma(R)},j_{\sigma(R)}+m}\cdots T_{j_{\sigma(1)},j_{\sigma(1)}+m}s
 % \end{equation*}
 % as well, and let $h_1',\dots,h_R'$ be the hooks in
 % $\lambda,\lambda\setminus\bs(h_1'),\dots,\lambda\setminus\bs(h_1')\setminus\dots\setminus\bs(h_{R-1}')$
 % corresponding to the pairs of indices
 % $(j_{\sigma(1)},j_{\sigma(1)}+m),\dots,(j_{\sigma(R)},j_{\sigma(R)}+m)$, respectively. Then
 % \begin{equation}\label{eq:heightmod}
  %  \sum_{i=1}^{R}\Ht(h_i)\equiv\sum_{i=1}^R\Ht(h_i')\pmod{2}.
  %\end{equation}
%\end{lemma}

\section{Proof of Lemma~\ref{lem6.2}} \label{seclem6.2}

Let $\lambda$ be a $p^{r-1}m$-core partition, and let $s$ be a representative of its abacus.  Let $s^{\prime}$ 
be the sequence obtained by removing a sequence of $R=p^{r-1}$ border strips of length $m$ from $\lambda$, and 
let $\lambda^{\prime}$ be the partition associated to $s^{\prime}$.   Our goal is to show that the number of ways of 
reaching $\lambda^{\prime}$ starting from $\lambda$ is a multiple of $p$.   

Let us first note that when $r=1$, it is impossible to remove a border strip of length $m$
from $\lambda$, since $\lambda$ is an $m$-core partition by assumption.  Thus the number of ways here
is $0$, and the lemma holds (vacuously).  Henceforth, assume that $r\ge 2$.

For each $a=1$, $\ldots$, $m$, consider the subsequences of $s$ and $s^{\prime}$ obtained by restricting to the progression 
$a\pmod m$: thus, set
$$ 
s(a;m) = (s(a+\ell m))_{\ell \in {\mathbb Z}}, \qquad s^{\prime}(a;m) = (s^{\prime}(a+\ell m))_{\ell \in {\mathbb Z}}. 
$$ 
We may think of $s(a;m)$ and $s^{\prime}(a;m)$ as corresponding to partitions $\lambda(a;m)$ and $\lambda^{\prime}(a;m)$, and note that a hook of length $m$ in 
the partition $\lambda$ corresponds to a hook of length $1$ (or simply a border square) in the partition $\lambda(a;m)$ 
(for some choice of $a$).  Since $\lambda^{\prime}(a;m)$ arises from $\lambda(a;m)$ by removing some number of hooks of length $1$, 
the Young diagram of the partition $\lambda^{\prime}(a;m)$ is contained in the Young diagram of the partition $\lambda(a;m)$
 (that is, $\lambda_i(a;m) \ge \lambda^{\prime}_i(a;m)$ for all $i$).  The difference between the Young diagram of $\lambda(a;m)$ and $\lambda^{\prime}(a;m)$ 
 (in other words, the boxes in $\lambda(a;m)$ that are not in $\lambda^{\prime}(a;m)$) is a skew diagram, which we denote by $\lambda(a;m)/\lambda^{\prime}(a;m)$.  
 Let $\ell_a$ denote the size of this skew diagram $|\lambda(a;m)/\lambda^{\prime}(a;m)|$, so that $\ell_a$ hooks of length $1$ must 
 be removed from $\lambda(a;m)$ to reach $\lambda^{\prime}(a;m)$.   Since a total of $R=p^{r-1}$ hooks of length $m$ are removed to go from $\lambda$ to $\lambda^{\prime}$, 
 note that 
 $$ 
 R= p^{r-1} = \sum_{a=1}^{m} \ell_a. 
 $$

The number of ways to go from $\lambda(a;m)$ to $\lambda^{\prime}(a;m)$ by removing successively $\ell_a$ hooks of length $1$ equals 
the number of standard Young tableaux of skew shape $\lambda(a;m)/\lambda^{\prime}(a;m)$, which we denote (in the usual notation) by 
$f_{\lambda(a;m)/\lambda^{\prime}(a;m)}$.   Recall that a standard Young tableau of this skew shape is a numbering of the boxes in the 
skew diagram using the numbers $1$ to $\ell_a$ such that the entries are increasing from left to right in each row, and increasing down each 
column. Each such tableau corresponds to a way of removing hooks, by removing boxes in descending order of their entries. 

We can now count the number of ways of going from $\lambda$ to $\lambda^{\prime}$ by removing $R$ hooks of length $m$.  Note that removing a 
hook from one subsequence $s(a;m)$ has no impact on the hooks in any of the other subsequences.  Therefore the desired number of ways 
to proceed from $\lambda$ to $\lambda^{\prime}$ equals 
$$ 
\binom{p^{r-1}}{\ell_1, \ell_2, \ldots, \ell_m} \prod_{a=1}^m f_{\lambda(a;m)/\lambda^{\prime}(a;m)}. 
$$
 
 The multinomial coefficient $\binom{p^{r-1}}{\ell_1, \ell_2, \ldots, \ell_m}$ is a multiple of $p$, except in the situation where 
 $\ell_a = p^{r-1}$ for some $a$ (and all other $\ell_j$ are $0$).  Thus we are left with the case when all the hooks of length $m$ 
 in going from $\lambda$ to $\lambda^{\prime}$ are confined to one subsequence $s(a;m)$.  So far, we have not made use of 
 the condition that $\lambda$ is a $p^{r-1}m$-core, and it is only in this case that we need this assumption.  The assumption 
 implies that $\lambda(a;m)$ is $p^{r-1}$-core, and so the skew diagram $\lambda(a;m)/\lambda^{\prime}(a;m)$ 
 (which has size $\ell_a = p^{r-1}$) cannot be a border strip of $\lambda(a;m)$.  In this situation, it turns out that $f_{\lambda(a;m)/\lambda^{\prime}(a;m)}$ 
 is a multiple of $p$.  This is implied by our next lemma, which is perhaps of independent interest.

\begin{lemma} \label{lem8.1}  Let $\pi$ and $\tau$ be two partitions, with the Young diagram of 
$\pi$ containing the Young diagram of $\tau$ (thus $\pi_i \ge \tau_i$ for all $i$).   Suppose the skew diagram 
$\pi/\tau$ is not a border strip of the partition $\pi$ (equivalently, either $\pi/\tau$ is disconnected, or it contains a 
$2\times 2$ square), and that $|\pi/\tau| = p^{t}$ is a prime power (with $t >0$).   Then the number of standard Young tableaux of skew shape $\pi/\tau$,  
denoted $f_{\pi/\tau}$, is a multiple of $p$. 
\end{lemma} 
\begin{proof}  First suppose that $\pi/\tau$ is disconnected, and is composed of $k\ge 2$ maximally connected skew shapes 
$S_1$, $\ldots$, $S_k$, with $|S_j| = s_j \ge 1$.  Then 
$$ 
f_{\pi/\tau} = \binom{p^t}{s_1, \ldots, s_k} f_{S_1} \cdots f_{S_k}, 
$$ 
is clearly a multiple of $p$. 

Now suppose that $\pi/\tau$ is a connected skew shape, but contains a $2\times 2$ square so that it is not a border strip of $\pi$.  
Since $f_{\pi/\tau}$ depends only on the shape $\pi/\tau$, we may assume that $\pi$ is minimal, having only as many rows and columns as 
needed for the skew shape $\pi/\tau$.  Then the maximal hook length of $\pi$ equals the number of border squares of $\pi$, which is strictly 
smaller than $|\pi/\tau| = p^{t}$ (since $\pi/\tau$ is not a border strip by assumption).   

It is a basic fact (see Section I.9 of~\cite{Macdonald1995}, for example --- the identity
below follows from equation~(9.1) of~\cite{Macdonald1995} by taking the
Hall inner product of both sides with the symmetric function $e_1^{p^t}$) that
\begin{equation*}
  f_{\pi/\tau}=\sum_{\nu\vdash p^{t}}f_{\nu}c^{\pi}_{\tau \nu},
\end{equation*}
where the sum is over partitions $\nu$ of $|\pi/\tau| = p^t$, $f_{\nu}=\chi^{\nu}_{(1,\dots,1)}$ is the degree of the irreducible character
corresponding to $\nu$ and the $c^{\pi}_{\tau \nu}$ are the Littlewood--Richardson coefficients (which are integers). By Lemma~\ref{lem:combiningparts},
$f_\nu\equiv\chi^{\nu}_{(p^{t})}\pmod{p}$, so that $p\mid f_\nu$ unless $\nu$ is a hook
of length $p^{t}$. Suppose now that $\nu$ is a hook of length $p^t$.  Here we use that the Littlewood--Richardson coefficient $c^{\pi}_{\tau \nu}$ 
is zero unless the Young diagram of the partition $\nu$ is contained in that of $\pi$ (see Section I.9 of~\cite{Macdonald1995} once again).  But 
all the hooks of $\pi$ have length $<p^t$, and therefore $\pi$ cannot contain a hook $\nu$ of length $p^t$.  Thus either $c^{\pi}_{\tau \nu}=0$ 
or $p | f_\nu$, and therefore the lemma follows. 
\end{proof}

\section{Preliminaries for the proof of Proposition~\ref{prop:manylargeparts}}\label{sec:prelims}

As in \cite{PeluseSoundararajan2022}, let $\widetilde{p}(k)$ denote the number of partitions of a
nonnegative integer $k$ into powers of $p$, with the convention that $\widetilde{p}(0)=1$. Denote by $F_p(t)$ the 
associated generating function
\begin{equation*}
  F_p(t):=\sum_{k=0}^{\infty}\widetilde{p}(k)e^{-k/t}=\prod_{j=0}^\infty\Big (1-e^{-p^j/t}\Big)^{-1},
\end{equation*}
where $t>0$ is a real number. We begin by recalling some estimates from our prior
work~\cite{PeluseSoundararajan2022}.

\begin{lemma}[Lemma~2 of~\cite{PeluseSoundararajan2022}]\label{lem:generatingfunctionbounds}
  When $0<t\leq 1$, we have $F_p(t)=O(1)$, and when $t\geq 1$, we have
  \begin{equation*}
    \frac{(\log{t})^2}{2\log{p}}+\frac{1}{2}\log{t}+O(1)\leq\log{F_p(t)}\leq\frac{(\log{t})^2}{2\log{p}}+\frac{1}{2}\log{t}+\frac{1}{8}\log{p}+O(1).
  \end{equation*}
\end{lemma}
More precise results are known for fixed primes $p$, as partitions into prime powers have
been studied extensively since the work of Mahler~\cite{Mahler1940} and de
Bruijn~\cite{deBruijn1948}. We will only require the estimates of
Lemma~\ref{lem:generatingfunctionbounds}, which are cruder but uniform in $p$.

Given a partition $\mu$ of $k$ into powers of $p$, let $\widetilde{\mu}$ denote the
partition obtained by repeatedly replacing every occurrence of $p^r$ parts of the same
size $p^j$ by $p^{r-1}$ parts of size $p^{j+1}$ until no part appears more than
$p^{r}-1$ times. For every nonnegative integer $s$, define $\widetilde{p}(k;s)$ to be the
number of partitions $\mu$ of $k$ into powers of $p$ such that $\widetilde{\mu}$ does not 
contain (at least) $p^{r-1}$ parts of the same size $p^j$ for any $j \ge s$.  The second lemma of this section gives a useful lower bound for
the difference between $\widetilde{p}(k)$ and $\widetilde{p}(k;s)$.

\begin{lemma}\label{lem:partitiondifference}
  For all $s\geq 2$ and $k\geq p^{r+s-1}(1+4/s)$, we have
  \begin{equation*}
    \widetilde{p}(k)-\widetilde{p}(k;s)\geq\frac{p^{s(s-1)/2}}{(s-1)^{s-1}}.
  \end{equation*}
\end{lemma}
\begin{proof}
  We will construct at least $p^{s(s-1)/2}/(s-1)^{s-1}$ partitions of $k$ counted in
  $\widetilde{p}(k)$ but not in $\widetilde{p}(k;s)$. For each $1\leq i\leq s-1$, pick an integer
  $a_i$ in the range
  \begin{equation*}
    0\leq a_i\leq\frac{p^{s-i}}{s-1}.
  \end{equation*}
  Each choice of $a_1,\dots,a_{s-1}$ gives a partition $\mu$ counted in $\widetilde{p}(k)$ by
  using $a_i$ copies of $p^i$ for $1\leq i\leq s-1$ and $k-\sum_{i=1}^{s-1}a_ip^i$ copies
  of $1$. The number of such partitions is
  \begin{equation*}
    \prod_{i=1}^{s-1}\Big\lceil\frac{p^{s-i}}{s-1} \Big\rceil\geq\prod_{i=1}^{s-1}\frac{p^{s-i}}{s-1}=\frac{p^{s(s-1)/2}}{(s-1)^{s-1}}.
  \end{equation*}
  Note that if $i>s-\log (s-1)/\log p$, then $a_i$ must be zero, so that all of these
  partitions have largest part at most $\frac{p^{s}}{s-1}$.

  We must check that each such $\mu$ is not counted in $\widetilde{p}(k;s)$;  that is, that the
  corresponding $\widetilde{\mu}$ contains at least $p^{r-1}$ copies of some part $p^j$ with $j\geq s$. Suppose that 
  this is not the case.  Notice that, by construction, the number of times any part appears in $\mu$ 
  is congruent modulo $p^{r-1}$ to the number of times it appears in ${\widetilde \mu}$.  
  Since no part can appear more than $p^{r}-1$ times in
  $\widetilde{\mu}$, it follows that any part that appears fewer than $p^{r-1}$ times or more
  than $p^r-p^{r-1}$ times in $\widetilde{\mu}$ must have appeared in the original partition
  $\mu$.   Since all the parts of $\mu$ are below $p^{s}/(s-1)$, we conclude that ${\widetilde \mu}$ 
  can contain (i) at most $p^{r}-1$ copies of any part $p^j$ with $p^j \le p^{s}/(s-1)$, (ii) at most 
  $p^r - p^{r-1}$ copies of any part $p^j$ with $p^s/(s-1) < p^j \le p^{s-1}$, and (iii) no parts of size $p^j$ 
  with $j \ge s$.  But these constraints imply that 
 \begin{align*} 
  k = |{\widetilde \mu}| &\le (p^r -1) \sum_{p^{j} \le p^s/(s-1)} p^j + (p^r-p^{r-1}) \sum_{p^{s}/(s-1) <p^j \le p^{s-1}} p^j \\
  &< (p^{r-1}-1) \frac{p^s}{(s-1)} \Big(1- \frac 1p\Big)^{-1}  + (p^r -p^{r-1}) p^{s-1} \Big(1- \frac 1p\Big)^{-1}  < p^{r+s-1} \Big(1 +\frac 4s\Big), 
  \end{align*}
  which contradicts our assumption on the size of $k$.  
%  The parts $p^j$ appearing in $\widetilde{\mu}$ but not in $\mu$ therefore contribute
  %at most $(p^{r}-p^{r-1})p^j$ each to $|\widetilde{\mu}|$, while the parts appearing in both
  %$\widetilde{\mu}$ and $\mu$ contribute at most an additional $(p^{r-1}-1)p^j$ to
  %$|\widetilde{\mu}|$. Since the largest part of $\widetilde{\mu}$ has size at most $p^{s-1}$ and
  %the largest part of $\mu$ has size at most $\frac{p^s}{s-1}$, it therefore follows that
  %\begin{align*}
  % |\widetilde \mu|=  k &<\left(p^{r}-p^{r-1}\right)\sum_{i=1}^{\infty}p^{s-i}+\left(p^{r-1}-1\right)\sum_{j=0}^\infty\frac{p^{s-j}}{s-1}
 %   \\
  %  &\leq
   %   p^{r+s-1}\Big(1-\frac{1}{p}\Big)\cdot\frac{1}{1-1/p}+\frac{p^{r+s-1}}{s-1}\cdot 2
  %  \\
  %  &\leq p^{r+s-1}\Big(1+\frac{4}{s}\Big).
 % \end{align*}
 % But this contradicts the assumption $k>p^{r+s-1}(1+4/s)$.
\end{proof}

\section{Proof of Proposition~\ref{prop:manylargeparts}}\label{sec:finalproof}

Let $\mathcal{L}$ be a set of positive integers coprime to $p$, and define
$p(n;\mathcal{L},s)$ to be the number of partitions $\mu$ of $n$ for which $\widetilde{\mu}$
contains fewer than $p^{r-1}$ parts of the same size $\ell p^j$ for every $\ell\in\mathcal{L}$
and $j\geq s$. We will prove Proposition~\ref{prop:manylargeparts} by obtaining an 
upper bound for $p(n;\mathcal{L},s)$ for well-chosen $\mathcal{L}$ and $s$.

\begin{lemma} \label{lem9.1}  Suppose that $n$ is large and $p^r \le 10^{-3} \log n/\log \log n$.  
Put 
\begin{equation} 
\label{9.1} 
x= \frac{\sqrt{6n}}{\pi}, \qquad s = \Big \lfloor \frac{\log \sqrt{n}}{ep^r}\Big\rfloor, 
\end{equation}
and let ${\mathcal L}$ be the set of integers in the interval $[L, L+x/p^{r+s-1}]$ that are 
coprime to $p$, where $L$ is a parameter lying in the range 
\begin{equation} 
\label{9.2} 
\frac{\sqrt{6n}}{2 \pi p^{r+s-1}} \le L \le \Big(1 +\frac{1}{5p^r}\Big) \frac{\sqrt{6n}}{2\pi p^{r+s-1}} \log n. 
\end{equation} 
Then 
$$ 
p(n;{\mathcal L},s) \ll p(n) n^{\frac 34} \exp(- n^{\frac{1}{16 p^r}}). 
$$
\end{lemma} 

Before proving the lemma, let us see how Proposition~\ref{prop:manylargeparts} would follow. 
Choose $r$ distinct values $L_j$ (with $1\le j\le r$) all in the range 
$$ 
\Big( 1+ \frac{1}{6p^r}\Big) \frac{\sqrt{6n}}{2\pi p^{r+s-1}} \log n \le L_j \le \Big( 1+ \frac {1}{5p^r}\Big) \frac{\sqrt{6n}}{2\pi p^{r+s-1}} \log n, 
$$ 
such that the corresponding sets ${\mathcal L_j}$ are all disjoint.  A partition $\mu$ for which ${\widetilde \mu}$ 
does not contain $r$ distinct parts $m_1$, $\ldots$, $m_r$ each appearing at least $p^{r-1}$ times and suitably large 
as desired in the proposition, must be counted among some $p(n;{\mathcal L}_j,s)$ with $1\le j\le r$.  Thus by Lemma \ref{lem9.1} 
the number of such bad partitions $\mu$ is 
$$ 
\le \sum_{j=1}^{r} p(n;{\mathcal L}_j,s) \ll r p(n)n^{\frac 34} \exp(- n^{\frac 1{16p^r}}) \ll p(n) n \exp(- n^{\frac 1{16p^r}}) \ll p(n) \exp(-n^{\frac{1}{20p^r}}),  
$$ 
as claimed.

\begin{proof}[Proof of Lemma \ref{lem9.1}] Consider the process of going from a partition $\mu$ to $\widetilde{\mu}$ by combining $p^{r}$
parts of the same size $m$ into $p^{r-1}$ parts of size $pm$. Suppose that $\ell$ is coprime
to $p$, and that the sum of all parts of the form $\ell p^j$ appearing in $\mu$ equals
$\ell k$. Restricting our attention to these parts, we may think of $\mu$ as giving rise
to a partition of $k$ into powers of $p$, and then $\widetilde{\mu}$ correspondingly gives a
partition of $k$ into powers of $p$ obtained by repeatedly combining $p^r$ parts of size
$p^j$ into $p^{r-1}$ parts of size $p^{j+1}$. It follows that $p(n;\mathcal{L},s)$ is the
coefficient of $z^n$ in the generating function
\begin{equation*}
  \prod_{\substack{\ell\notin\mathcal{L} \\ (\ell,p)=1}}\prod_{j=0}^\infty \Big(1-z^{\ell
      p^j}\Big)^{-1}\prod_{\ell\in\mathcal{L}}\Big( \sum_{k=0}^{\infty}\widetilde{p}(k;s)z^{\ell
    k}\Big),
\end{equation*}
which equals
\begin{equation*}
  \prod_{i=1}^\infty\left(1-z^i\right)^{-1}\prod_{\ell\in\mathcal{L}}\Big(\frac{\sum_{k=0}^\infty\widetilde{p}(k;s)z^{\ell
      k}}{\sum_{k=0}^\infty\widetilde{p}(k)z^{\ell k}}\Big).
\end{equation*}
Since all of the coefficients in the generating function for $p(n;\mathcal{L},s)$ are
nonnegative, we must have, for any $0< z<1$, 
\begin{equation}\label{eq:generatingfunctionbound}
  p(n;\mathcal{L},s)\leq  \frac{1}{z^n}\prod_{i=1}^\infty\left(1-z^i\right)^{-1}\prod_{\ell\in\mathcal{L}}\Big(\frac{\sum_{k=0}^\infty\widetilde{p}(k;s)z^{\ell
      k}}{\sum_{k=0}^\infty\widetilde{p}(k)z^{\ell k}}\Big). 
\end{equation}

Recall that $x= \sqrt{6n}/\pi$, and take $z= e^{-1/x}$ in the bound~\eqref{eq:generatingfunctionbound}.  
Then, by the asymptotic formula for
the partition function and basic estimates for the generating function of the number of partitions
(see Section VIII.6 of~\cite{FlajoletSedgewick2009}), we obtain
\begin{equation}
\label{9.4} 
  p(n;\mathcal{L},s)\ll n^{3/4}p(n)\prod_{\ell\in \mathcal{L}}\Big(\frac{\sum_{k=0}^\infty\widetilde{p}(k;s)z^{\ell
      k}}{\sum_{k=0}^\infty\widetilde{p}(k)z^{\ell k}}\Big) \ll n^{3/4}p(n)\exp(-\Delta),
\end{equation}
where
\begin{equation*}
  \Delta:=\sum_{\ell\in\mathcal{L}}\frac{1}{F_p(x/\ell)}\sum_{k=0}^\infty(\widetilde{p}(k)-\widetilde{p}(k;s))e^{-\ell k/x}.
\end{equation*}

Our work so far applies to any set ${\mathcal L}$ of integers that are coprime to $p$, and we now proceed to the situation at hand. 
The lower bound on $L$ and our choice of $x$ give, for all $\ell \in {\mathcal L}$,  the bound 
\begin{equation*}
  F_p\Big(\frac{x}{\ell}\Big) \leq F_p\Big(\frac{x}{L}\Big)\leq F_p\Big( \frac{2p^{r+s-1}}{\log{n}}\Big).
\end{equation*}
From this estimate, our choice of ${\mathcal L}$, and Lemma~\ref{lem:partitiondifference} it follows  that
\begin{equation}\label{eq:Deltalowerbound}
  \Delta\geq \frac{1}{F_p(2 p^{r+s-1}/\log{n})}\sum_{\substack{L\leq\ell\leq
      L+x/p^{r+s-1} \\ (\ell,p)=1}}\ \ \sum_{k\geq
    p^{r+s-1}(1+4/s)}\frac{p^{s(s-1)/2}}{(s-1)^{s-1}}e^{-\ell k/x}. 
\end{equation}

For $\ell$ in the range $L\le \ell \le L+x/p^{r+s-1}$, we have
\begin{align*}
  \sum_{k\geq p^{r+s-1}(1+4/s)}e^{-\ell k/x} &\geq \exp\Big(-\frac{\ell
                                               p^{r+s-1}}{x}\Big(1+\frac{4}{s}\Big)\Big) \frac{e^{-\ell/x}}{1- e^{-\ell/x}} 
  \\
  &\geq \frac{x}{2 L}\exp\Big(-\Big(\frac{L p^{r+s-1}}{x}+1\Big)\Big(1+\frac{4}{s}\Big)\Big).
\end{align*}
%where we have used the upper bound $\ell\leq L+x/p^{r+s-1}$. 
Inserting this into the right-hand side of~\eqref{eq:Deltalowerbound} and 
noting that (since $p^{r+s-1}$ is small in comparison to $x$)
$$
|{\mathcal L}| \ge \Big(1-\frac 1p\Big) \frac{ x}{p^{r+s-1} } - 2 \ge \frac{x}{3p^{r+s-1}},
$$  
we obtain (using our choice of $x$ and the range for $L$)
\begin{align*}
  \Delta &\geq \frac{1}{F_p(2 p^{r+s-1}/\log{n})}\cdot
           \frac{p^{s(s-1)/2}}{(s-1)^{s-1}}\cdot\frac{x}{3p^{r+s-1}}\cdot\frac{x}{2L} \exp\Big(-\Big(\frac{L
           p^{r+s-1}}{x}+1\Big)\Big(1+\frac{4}{s}\Big)\Big)
  \\
  &\geq \frac{1}{6}\frac{1}{F_p(2 p^{r+s-1}/\log{n})}\cdot
           \frac{p^{s(s-1)/2}}{(s-1)^{s-1}}\cdot\frac{x}{\log{n}}\cdot \exp\Big(-\Big(\frac{L
    p^{r+s-1}}{x}+1\Big)\Big(1+\frac{4}{s}\Big)\Big).
\end{align*}

%We now obtain a good lower bound on $\Delta$ for suitable $\mathcal{L}$ and $s$.  It is convenient to specify immediately 
%our choice of $s$: 
%\begin{equation} 
%\label{9.3} 
%s=\Big \lfloor \frac{\log \sqrt{n}}{e p^{r}} \Big \rfloor.
%\end{equation} 
% In proving Proposition~\ref{prop:manylargeparts} we may suppose that $p^{r} \le  10^{-3}\log n/(\log \log n)^2$,  so that 
% $p^{r+s-1} \le n^{\frac 14}$ is small in comparison with $x$.  

%Let $\mathcal{L}$ be the set of integers in the interval $[L,L+ x/p^{r+s-1}]$ 
%that are coprime to $p$, where $L$ is a parameter lying in the range
%\begin{equation*}
%  \frac{\sqrt{6n}}{2\pi p^{r+s-1}}\log{n}\leq L\leq 2\cdot \frac{\sqrt{6n}}{2\pi p^{r+s-1}}\log{n}.
%\end{equation*}
%for our choice of $\mathcal{L}$.

%say, by our choice of $x$ and the bounds on $L$.

Using Lemma~\ref{lem:generatingfunctionbounds} and the bound $p^r \le \log \sqrt{n}$, it follows that
\begin{align*}
  \log{F_p\Big(\frac{2p^{r+s-1}}{\log{n}}\Big)}
  &\leq
                               \frac{1}{2\log{p}}\Big(\log\frac{p^{r+s-1}}{\log{\sqrt{n}}}\Big)^2+\frac{1}{2}\log\frac{p^{r+s-1}}{\log{\sqrt{n}}}+\frac{1}{8}\log{p}+O(1)
  \\
  &\leq \frac{1}{2\log{p}}\Big(\log\frac{p^{r+s-1}}{\log{\sqrt{n}}}\Big)^2+\frac{s}{2}\log{p}+O(1).
\end{align*}
Therefore 
\begin{align*}
  \log{\frac{p^{s(s-1)/2}}{F_p(2 p^{r+s-1}/\log{n})(s-1)^{s-1}}} &\geq
                                                                   \frac{s^2}{2}\log{p}-\frac{1}{2\log{p}}\Big(\log\frac{p^{r+s-1}}{\log{\sqrt{n}}}\Big)^2-s\log{ps}+O(1)
  \\
  &\geq s\log{\frac{\log{\sqrt{n}}}{p^rs}}-\frac{\left(\log\log{\sqrt{n}}\right)^2}{2\log{p}}+O(1).
\end{align*}

Recalling our choice of $s$, we conclude that 
\begin{align*}
 \log \Delta &\ge  s\log{\frac{\log{\sqrt{n}}}{p^rs}}+\log{\sqrt{n}}-\frac{(\log\log{\sqrt{n}})^2}{2\log{p}}-\log\log{n}-\frac{L
    p^{r+s-1}}{x}\Big(1+\frac{4}{s}\Big)+O(1) \\ 
&\geq \Big(1+\frac{1}{ep^r}\Big)\log\sqrt{n}-\frac{L p^{r+s-1}}{x}-\log\log{n}-\frac{\left(\log\log{\sqrt{n}}\right)^2}{2\log{p}}+O(1)\\
&\ge \Big( \frac{1}{ep^r} - \frac{1}{5p^r}\Big) \log \sqrt{n} - (\log \log n)^2 + O(1) \ge \frac{\log n}{15p^r} - (\log \log n)^2 +O(1), 
\end{align*}
upon using the upper bound on $L$ in \eqref{9.2}.  In the range $p^r\leq 10^{-3}\log{n}/(\log\log{n})^2$ we find 
$$ 
\log \Delta \ge \frac{\log n}{16 p^r} + O(1), 
$$ 
which when used in \eqref{9.4} yields the lemma. 
 \end{proof}

\bibliographystyle{plain}
\bibliography{bib}

\end{document}